\documentclass[12pt]{amsart}
\usepackage{geometry}               
\usepackage{amssymb}
\usepackage{mathrsfs}
\usepackage{epstopdf}
\usepackage{amsmath}
\usepackage{hyperref}
\usepackage{color}
\usepackage{amsthm}
\usepackage{amssymb}
\usepackage[english]{babel}
\usepackage{blindtext}
\usepackage{indentfirst}

\theoremstyle{definition}
\newtheorem{theorem}{Theorem}[section]
\newtheorem{definition}[theorem]{Definition}
\newtheorem{corollary}[theorem]{Corollary}
\newtheorem{lemma}[theorem]{Lemma}
\newtheorem{proposition}[theorem]{Proposition}
\newtheorem{remark}{Remark}
\newtheorem{example}[theorem]{Example}
\numberwithin{equation}{section}

\begin{document}

\title[Finsler perturbation with nondense irrational geodesics]{Finsler perturbation with nondense geodesics with irrational directions}

\author{Dmitri Burago and Dong Chen}
\address{Dmitri Burago: Department of Mathematics,  Pennsylvania State University, University Park, PA 16802, USA}
\email{burago@math.psu.edu}

\address{Dong Chen: Department of Mathematics,  Pennsylvania State University,  University Park, PA 16802, USA}
\email{dxc360@psu.edu}
\date{}

\thanks{The First author was partially supported
by NSF grant DMS-1205597. The Second author was partially 
supported by NSF grant DMS-1205597 and D. Burago's PSU Research Fund.}



\maketitle
\begin{abstract}
 We show that given any Liouville direction and flat Finsler torus, one can  make a $C^{\infty}$-small perturbation on an arbitrarily small disc to get a nondense geodesic in the given direction.
\end{abstract}

\section{Introduction}

Let $M$ be a closed smooth manifold with universal cover $\tilde{M}$. A 
geodesic $\gamma$ in $\tilde{M}$ is called \textit{minimal} if it realizes the distance 
between any two points $p,q\in\gamma$. If all geodesics in $\tilde{M}$ are minimal, 
then we say $M$ \textit{has no conjugate points}.  In the 1940s, Hedlund and Morse \cite{HM} 
asked the following question: are Riemannian tori without conjugate points flat? A few years later, E. Hopf \cite{Hp} 
gave a positive answer in the 2-dimensional case and as it often happens, the higher dimensional version is now known as the E. Hopf Conjecture,  though seemingly Hopf never conjectured that and it apparently goes back to Hedlund and Morse.
After Hopf's result, many other people 
studied this problem under various assumptions, see e.g. \cite{CK}. 
%Croke conformal metrics, Green dim 2
Finally, in 1994,  almost half a 
century after Hopf's result,  the conjecture has been proven in  \cite{BI94}.

However,  as we turn our attention to Finsler manifolds,   we find a different world.  
Non-flat Finsler tori without conjugate points can be constructed by making symplectic 
(contact) perturbations of flat Riemannian metrics \cite{K} or as some metrics of revolution \cite{Z}. 
Moreover, any sufficiently small region in any Finsler surface can be isometrically embedded into some Finsler 2-torus without conjugate points \cite{C2}. This means the local structure of such tori is 
totally flexible,  contrary to the rigidity suggested in the Hopf Conjecture.

A more reasonable analog of the Hopf Conjecture in the Finsler setting is that the geodesic flow on any Finsler torus without conjugate points is smoothly conjugate to that on some flat Finsler torus. This conjecture is still open and it is equivalent to the Hopf Conjecture in the Riemannian setting \cite{CK}.
Up to a time change, this conjecture is equivalent to the smoothness of the Heber foliation \cite{Hb} for Finsler tori without conjugate points. There are two possible reasons why Heber foliation may fail to be smooth: (1) There may be one individual leaf on which the dynamics is not smoothly conjugate to a linear flow,  
 %Since the flow is i close it has a nice po section near the circle 
%There may be one individual torus where the dynamics is not smoothly conjugate to a linear flow 
or (2)  The leaves are smooth but they behave in a non-smooth way in the transverse direction.  The possibility of neither situation is known so far. In this paper we are trying to approach some understanding of situation (1).  
  
To better understand (1) and to facilitate understanding the conditions of 
Proposition \ref{mainprop1}, 
let us consider a 2-D flat Finsler torus $\mathbb{T}^2$ with 
the standard 
coordinates $(q_1, q_2, v_1, v_2)$.  Later on, we use $(p_1,p_2)$ for the momenta in the cotangent bundle so that our notations agree with those in  \cite{BCI}. However, here, for the sake of visualization,  we work in the phase space of the original Lagrangian system.
To have a nice Poincar\'{e} 
section, we use only
perturbations that do not change lengths of unit tangent vectors with, say, 
$|v_1/v_2|\geq 1$. 
This means that the metric and the flow do not change on the part of the unit tangent
bundle \{$(\mathbf q, \mathbf v)$\} constituted by the vectors 
that form angles $\leq \pi/4$ (the angle is measured with respect to the standard metric on the torus with coordinates $(\mathbf q, \mathbf v)$)  with the circles
$\{q_2=const\}$. %: \angle (\mathbf v, \frac\partial {\partial q_1}) \leq \pi/4\}$. 
Then the part $U$
of the unit tangent bundle with $\angle (\mathbf v, \frac\partial {\partial q_1}) \geq \pi/4$ is also invariant under the perturbed flow and, therefore, for sufficiently small perturbations,
the hypersurface (with boundary) \{$(q_1, 0, v_1, v_2): |v_1/v_2|\leq 1$\}
is a Poincar\'{e} section for the perturbed flow restricted to $U$. Now, let us
fix a rotation direction $\alpha$ which is sufficiently close to the vertical one,
$\frac\partial{\partial q_2}$ (namely, forming the angle smaller than $\pi/4$ with 
this direction).  Now we can try to define a map $\tilde F$ from the circle $\Sigma:=\{(q_1, q_2): q_2=0\}$ to itself by 
starting from any  point $\mathbf q$ from $\Sigma$, following the geodesic 
with the rotation vector $\alpha$ until its next intersection with $\Sigma$, and 
pronouncing this intersection point $\tilde F (\mathbf q)$ the image of $\mathbf q$ (we use $\tilde{}$ to emphasize that we are dealing with a perturbed flow).
%(we use $\tilde P$ to reserve $P$ for the same map for the unperturbed flow).
Note that there is no reason to believe that this actually is a map and not 
a one-to-many
correspondence. We do not use this ``map" in the course of the proof, furthermore, 
even afterwards, we do not prove its existence, we have only a partially defined map arising from minimizing  $ \widetilde{\mathbb{T}}^2$-geodesics. As we see {\it a posteriori},
even if this 
map is
correctly defined, in our examples it cannot be smooth. 
If $\alpha$ has irrational slope and we are in situation (1) with a geodesic 
with irrational rotation direction and which is not dense on the torus, 
this ``map" is a 
Denjoy example 
(a $C^1$ counterexample to the Denjoy Theorem when the bounded 
variation condition is not satisfied,  see \cite{KH}, p. 403), and vice versa. 

 So far, we do not even know the answer to the following question:

\begin{center}
\textbf{Question:} Let $(\mathbb{T}^n, \varphi)$ be a Finsler torus without conjugate points,  and let $\gamma$ be a geodesic with irrational rotation vector.  Is it true that 
$\gamma$ is dense in $\mathbb{T}^n$?
\end{center}

In this paper we deal with an approximative version of the above question. To be more specific, we make a $C^{\infty}$- small perturbation of any flat torus to get a nondense geodesic whose rotation vector points at any given  Liouville direction.  Recall that a rationally independent vector is called \textit{Liouville} if it is not Diophantine, and a vector $v\in\mathbb{R}^n$ is called \textit{Diophantine} if there exist $\gamma, \tau>0$ such that $|v\cdot k|\geq \gamma|k|^{-\tau}$ for all $k\in\mathbb{Z}^n\backslash \{\textbf{0}\}$.

\begin{theorem}\label{thm1}
For any flat Finsler torus $(\mathbb{T}^2, \varphi_0)$ and any Liouville number $\omega$, one can make a $C^{\infty}$-small perturbation on  $\varphi_0$ in the class of Finsler metric so that the resulting metric has a nondense geodesic with rotation vector colinear to $(\omega, 1)$.  If $\varphi_0$ is reversible, the resulting metric can be made reversible as well.
\end{theorem}

%\noindent{\bf Remark}. 
\begin{remark}
Although we formulate and prove the theorem for flat 2-tori,  similar proof works with minor modification for flat $\mathbb{T}^n=\mathbb{T}^{n-1}\times S^1 (n\geq 3)$ as well.  To have less cumbersome notations,   we only consider the 2-dimensional case in this paper.  Moreover,  the resulting Finsler torus has conjugate points, see Remark \ref{rem: conj_pts}.
\end{remark}

%\noindent{\bf Remark}. 
%\begin{remark}
%Notice that the geodesic flow on the cotangent bundle of a flat torus is a completely integrable Hamiltonian flow, therefore by the celebrated KAM Theorem, any $C^{\infty}$-small perturbation will inevitably result in persistence of invariant tori on which the dynamics is conjugate to a linear flow with a Diophantine rotation vector (see e.g. \cite[Theorem 6.11]{dlL}).  Every trajectory on such a torus projects to a dense orbit on the base torus. Hence Theorem \ref{thm1} is optimal in the sense that the only possible non-dense irrational geodesics could only be those with Liouville rotation vectors.  (Change???)
%\end{remark}

The paper is organized as follows.  In Section 2 we go over some background on Finsler manifold and results from dual lens map techniques. In Section 3 we cover some basic terminology and properties of twist maps and minimal configurations,  while in Section 4 we give an extension of Mather's destruction of invariant tori with Liouville rotation number \cite{M88} so that the perturbation of the generating function is supported on a small region on the cylinder. The results from Section 2 and 4 are combined to give a proof of Theorem \ref{thm1} in Section 5.

Here is a sketch of the proof. Firstly we show that the Poincar\'{e} map of the unperturbed geodesic flow is conjugate to $R_1: S^1\times (-\Lambda,1) \rightarrow S^1\times (-\Lambda,1)$ defined by
$$R_1(q_1, p_1)=(q_1+f'_0(p_1), p_1),$$
where $f_0$ is a smooth function defined near $0$. Secondly, in Lemma \ref{lem1} and Lemma \ref{lem2} we verify the twist conditions (see Definition 3.1) for $R_1$ and the conditions $(H_1)-(H_{6\theta})$ for $h$, a generating function of $R_1$. This allows us to use Proposition \ref{mainprop2} to get a $C^{\infty}$-small perturbation $\tilde{R}_1$ of $R_1$ and a nondense $\tilde{R}_1$-orbit with the given Liouville rotation number. Finally, Proposition \ref{mainprop1} is applied to obtain a Finsler metric whose Poincar\'{e} map is conjugate to $\tilde{R}_1$.

{\it Acknowledgments}. The authors are grateful to Victor Bangert and Federico Rodriguez Hertz for useful discussions. The authors would also like to express their gratitude to referees for valuable suggestions on the improvement of the paper.

\section{Simple Finsler metrics and dual lens maps}

We use some notation and techniques from \cite{BI10}, \cite{BI16}, \cite{C1} and \cite{C2}. To make this note more self-contained and reader-friendly, we copy them here.

\subsection{Finsler metrics and geodesics}
A \textit{Finsler metric} $\varphi$ on $M$ is a smooth family of quadratically convex norms $\varphi(x, \cdot)$ on each tangent space $T_x M$. It is \textit{reversible} if $\varphi(x,v)=\varphi(x,-v)$ for all $x\in M, v\in T_x M$. A \textit{unit sphere} (or \textit{indicatrix}) in $T_x M$ is defined to be the collection of all vectors $v\in T_xM$ with $\varphi(x,v)=1$. We denote by $UTM$ the unit tangent bundle of $(M,\varphi)$.

Let $\gamma:[a,b]\rightarrow M$ be a smooth curve on a Finsler manifold $(M, \varphi)$. We may assume that $\gamma$ is \textit{unit-speed}, namely, $\varphi(\gamma(s), \gamma'(s))=1$ for all $s\in[a,b]$.  The length $L(\gamma)$ of $\gamma$ is simply $b-a$.  The \textit{distance function} $d_{\varphi}: M\times M\to \mathbb{R}$ is defined via
$d_{\varphi}(x,y):=\inf_{\gamma} L(\gamma)$, where the infimum is taken over all smooth curves starting at $x$ and ending at $y$.  $d_{\varphi}$ could be non-symmetric if $\varphi$ is not reversible. A unit-speed curve $\gamma: [a, b]\rightarrow M$ is called \textit{minimal} if for any $a\leq t_1<t_2\leq b, d_{\varphi}(\gamma(t_1), \gamma(t_2))=t_2-t_1$.  A locally minimal curve is called
a \textit{geodesic}.

The dual norm $\varphi^*$  on the cotangent bundle $T^*M$ of $M$ is defined by
\begin{equation}\label{eq: def_dual_norm}
\varphi^*(\alpha):=\sup_{v\in UT_xM}\{\alpha(v)\}, \text{ where }\alpha\in T^*_x M,
\end{equation}
and denote by $UT^*M$ the unit cotangent bundle. Let $\mathscr{L}: TM\rightarrow T^*M$ be the Legendre transform of the Lagrangian $\varphi^2/2$. It maps $UTM$ to $UT^*M$. For any tangent vector $v\in UT_xM$, its Legendre transform $\mathscr{L}(v)$ is the unique covector $\alpha\in U_x^*M$ such that $\alpha(v)=1$.

\subsection{Simple manifolds and Dual lens maps}
A Finsler $n$-disc $(D,\varphi)$ is called \textit{simple} if it satisfies the following three conditions:

(1) Every pair of points in $D$ is connected by a unique geodesic.

(2) Geodesics depend smoothly on their endpoints.

(3) $\partial D$ is strictly convex, that is, geodesics never touch it at their interior points.

Denote by $U_{in}$ (resp.  $U_{out}$) the set of unit tangent vectors with base points at the boundary $S:=\partial D$ and pointing inwards (resp.  outwards).  For any vector $v\in U_{in}$, we look at the geodesic with initial velocity $v$. Once it hits the boundary again, we get the exiting vector $\beta(v)\in U_{out}$.  This defines the \textit{lens map} $\beta: U_{in}\rightarrow U_{out}$. The \textit{dual lens map} is then defined by $\sigma:=\mathscr{L}\circ \beta\circ \mathscr{L}^{-1}: U^*_{in}\to U^*_{out}$, where  $U^*_{in}:=\mathscr{L}(U_{in})$ and $U^*_{out}:=\mathscr{L}(U_{out})$. If $\varphi$ is reversible then $\sigma$ is symmetric in the sense that $-\sigma(-\sigma(\alpha))=\alpha$ for all $\alpha\in U^*_{in}$. 

Note that $U^*_{in}$ and $U^*_{out}$ are $(2n-2)$-dimensional submanifolds of $T^*D$. The restriction of the canonical symplectic 2-form of $T^*D$ to $U^*_{in}$ and $U^*_{out}$ are nondegenerate hence  the symplectic structure. One can check that the dual lens map $\sigma$ is symplectic.

\subsection{Perturbation of Dual Lens Maps}
Under certain natural restrictions, a symplectic perturbation of $\sigma$ is the dual lens map of some metric that is close to $\varphi$:

\begin{theorem}[Burago-Ivanov \cite{BI16}]\label{thm3}
Assume that $n\geq 3$. Let $\varphi$ be a simple metric on $D=D^n$ and $\sigma$ its dual lens map. Let $W$ be the complement of a compact set in $U^*_{in}$. Then every sufficiently small symplectic perturbation $\tilde{\sigma}$ of $\sigma$ such that $\tilde{\sigma}|_W=\sigma|_W$ can be realized by the dual lens map of a simple metric $\tilde{\varphi}$ which coincides with $\varphi$ in some neighborhood of $\partial D$. The choice of $\tilde{\varphi}$ can be made in such a way that $\tilde{\varphi}$ converges to $\varphi$ whenever $\tilde{\sigma}$ converges to $\sigma$ (in $C^{\infty}$). In addition, if $\varphi$ is a reversible Finsler metric and $\tilde{\sigma}$ is symmetric then $\tilde{\varphi}$ can be chosen reversible as well.

\end{theorem}

When $n=2$, due to some topological obstructions, Theorem \ref{thm3} holds under additional conditions. 

\begin{definition}
For any symplectic map $\sigma: U^*_{in}\rightarrow U^*_{out}$ we can define two maps $P_{\sigma}: U^*_{in}\rightarrow S\times S$ and $Q_{\sigma}: U^*_{out}\rightarrow S\times S$ by 
$$P_{\sigma}(\alpha):=(\pi(\alpha), \pi(\sigma(\alpha)))$$
and
$$Q_{\sigma}(\beta):=(\pi(\sigma^{-1}(\beta)), \pi(\beta)),$$
here $\pi: T^*D\rightarrow D$ is the bundle projection. 
\end{definition}
\begin{remark}
Note that $P_{\sigma}=Q_{\sigma}\circ\sigma$ and both maps are bijections.  $P^{-1}_{\sigma}$($Q^{-1}_{\sigma}$ respectively) maps two different points on the boundary $S$ to the inwards (outwards respectively) covector of the geodesic connecting these two given points.
\end{remark}
\begin{definition}
Let $\Delta:=\{(x,x)\in S\times S: x\in S\}$. We define a 1-form $\lambda_{\sigma}$ on $S\times S\backslash\Delta$ as follows. For $p,q\in S, p\neq q, \xi\in T_p S, \eta\in T_q S$, define
$$\lambda_{\sigma}(\xi,\eta):=-P^{-1}_{\sigma}(p,q)(\xi)+Q^{-1}_{\sigma}(p,q)(\eta).$$
\end{definition}
\begin{proposition}[\cite{BI16}]\label{prop1}
Let $\sigma$ be the dual lens map of a simple Finsler metric $\varphi$ on $D^2$. Let $W$ be the complement of a compact set in $U^*_{in}$ and  $\tilde{\sigma}$ is a symplectic perturbation of $\sigma$ with $\tilde{\sigma}|_W=\sigma|_W$. Then $\tilde{\sigma}$ is the dual lens map of a simple metric $\tilde{\varphi}$ if and only if
$$\int_{\{p\}\times\{S\backslash p\}}\lambda_{\tilde{\sigma}}=0 $$
for some (and then all) $p\in S$. Convergence and reversible cases are the same as those in Theorem \ref{thm3}.
\end{proposition}
For each $p\in S$, let $(U^*_{in})_p$ be the fiber of $U^*_{in}$ over $p$.  We have the following corollary by Theorem \ref{thm3} and Proposition \ref{prop1}:
\begin{corollary}\label{maincor}
Let $\sigma$ be the dual lens map of a simple Finsler metric on $D^n$ ($n\geq 2$). Let $W$ be the complement of a compact set in $U^*_{in}$ and  $\tilde{\sigma}$ a symplectic perturbation of $\sigma$ with $\tilde{\sigma}|_W=\sigma|_W$. If there is an open subset $O \subseteq \partial D^n$ such that for each $p\in O$, 
$$(U^*_{in})_p\subseteq W, \eqno(\ast)$$
then $\tilde{\sigma}$ is the dual lens map of a simple metric $\tilde{\varphi}$. Convergence and reversible cases are the same as those in Theorem \ref{thm3}.
\end{corollary}
\begin{proof}
We have only to prove that for any $p\in O$ and $\xi\in (U^*_{in})_p$, 
$$\lambda_{\sigma}(\xi, \cdot)=\lambda_{\tilde{\sigma}}(\xi, \cdot).$$
Notice that on $(U^*_{in})_p$ we have $\sigma=\tilde{\sigma}$ hence $P_{\sigma}=P_{\tilde{\sigma}}$. Thus $P^{-1}_{\sigma}(p,\cdot)=P^{-1}_{\tilde{\sigma}}(p,\cdot)$. Similarly we have $Q^{-1}_{\sigma}(p,\cdot)=Q^{-1}_{\tilde{\sigma}}(p,\cdot)$, hence $\lambda_{\sigma}(\xi, \cdot)=\lambda_{\tilde{\sigma}}(\xi, \cdot)$ for all $\xi\in (U^*_{in})_p$.
\end{proof}

\subsection{Perturbation of flat metrics on $\mathbb{T}^n (n\geq 2)$}\label{sec: perturb_torus}
Let $(\mathbb{T}^n, \varphi_0)$ be a flat Finsler torus. The cotangent bundle $T^*\mathbb{T}^{n}$ is endowed with the canonical action-angle coordinates $(q_1, ..., q_n,p_1,..., p_n)$ of the geodesic flow. We think of $\mathbb{T}^n$ as the cube $[-1/2,1/2]^n$ with opposite sides identified. We are aiming at realizing any small symplectic perturbation of the Poincar\'{e} map supported on a small compact set as a result of a perturbation of the metric on $\mathbb{T}^n$.

Denote by 
$$\Psi(p_1,...,p_n):=\frac{\varphi_0^*(p_1,...,p_n)^2}{2}, \Psi_i:=\partial{\Psi}/\partial p_i, \text{ for }i=1,\cdots,n, $$
where $\varphi_0^*$ is the dual norm of $\varphi_0$ defined in \eqref{eq: def_dual_norm}. The geodesic flow on $T^*\mathbb{T}^n$ satisfies 
$$\dot{q}_i=\Psi_i, \dot{p}_i=0, \text{ for }i=1,\cdots,n, $$
Let $\textbf{q}=(q_1, ..., q_{n-1}), \textbf{p}=(p_1, ..., p_{n-1})$. Take a submanifold $T_0:=\{q_n=-1/2\}$ and a section 
$$\Gamma_0:=\{(\textbf{q}, q_n,\textbf{p}, p_n)\in UT^*\mathbb{T}^n: q_n=-1/2, \Psi_n>0\}.$$   
$\Gamma_0$ inherits a natural symplectic form from $T^*\mathbb{T}^n$.  By the Implicit Function Theorem,  for any covector in $\Gamma_0$, $p_n$ is a smooth function of $\textbf{p}$ with domain $U_p\subseteq \mathbb{R}^{n-1}$. More specifically, we have
$$p_n=-f_0(\textbf{p}), \textbf{p}\in U_p$$
and 
$$(f_0)_i(\textbf{p})=\frac{\Psi_i}{\Psi_n}(\textbf{p}, -f_0(\textbf{p})), \text{ for }i=1,\cdots,n.$$

Let $\Pi: \Gamma_0\to T_0\times U_p$ be the canonical projection defined by
$$\Pi(\textbf{q}, -1/2, \textbf{p}, p_n)=(\textbf{q}, \textbf{p}).$$
$\Pi$ is symplectic and it is a bijection. 
Denote by $R_0: \Gamma_0\rightarrow \Gamma_0$ the Poincar\'{e} return map to $\Gamma_0$ of the geodesic flow.  Define 
$$R_1:=\Pi\circ R_0\circ \Pi^{-1}: T_0\times U_p \rightarrow T_0\times U_p.$$
Since $R_0$ is symplectic, so does $R_1$. By equipping $T_0$ with an affine structure induced from $\mathbb{R}^{n-1}$,  a simple calculation gives us the following expression of $R_1$:
$$R_1(\textbf{q},\textbf{p})=\left(\textbf{q}+\nabla f_0(\textbf{p}), \textbf{p}\right) .$$

Denote by $\textbf{q}^2=\sum_{i=1}^{n-1}q^2_i$, $B(r):=\{\textbf{q}: \textbf{q}^2\leq r^2\}$ and $\Pi_q$ the canonical projection $(\textbf{q}, \textbf{p})\mapsto \textbf{q}$. We have the following analogue of Theorem \ref{thm3} for $\mathbb{T}^n$:

\begin{proposition}\label{mainprop1}
Let $(\mathbb{T}^n, \varphi_0)$ be a flat torus and $K\subseteq T_0\times U_p$ a compact set. If there exists $R>0$ such that $\Pi_q(K)\subseteq B(R)$ and $\Pi_q(R_1(K))\subseteq B(R^{-1})$, then any $C^{\infty}$-small symplectic perturbation $\tilde{R}_1$ of $R_1$ with $\tilde{R}_1|_{K^c}=R_1|_{K^c}$  is conjugated via $\Pi$ to the Poincar\'{e} return map to $\Gamma_0$ of the geodesic flow on some Finsler manifold $(\mathbb{T}^n, \tilde{\varphi})$. Convergence and reversible cases are the same as those in Theorem \ref{thm3}.
\end{proposition}
\begin{remark}
We formulate Proposition \ref{mainprop1} for a family of compact sets $K$.  In practice, we use specific $K$ similar to a subset of the cone field in the fourth paragraph of the introduction.
\end{remark}
\begin{proof}
Let $D_{1/2}:=\{\sum_{i=1}^n q_i^2 \leq 1/4\}$ be the ball in the cube and $\sigma: U^*_{in}\rightarrow U^*_{out}$ the dual lens map of the Finsler disc $(D_{1/2}, \varphi_0)$. We only change the metric inside $D_{1/2}$. For any $\alpha_1\in \Pi^{-1}(K)$ (resp. $\alpha_2\in \Pi^{-1}(R_1(K))$), consider its forward orbit (resp. backward orbit) under the geodesic flow generated by $\varphi_0$. Since $\Pi_q(K)\subseteq B(R)$ and $\Pi_q(R_1(K))\subseteq B(R^{-1})$, the forward orbit  (resp. backward orbit) will intersect $U^*_{in}$ (resp. $U^*_{out}$) transversally and we denote the intersection by $\phi_1(\alpha_1)$ (resp. $\phi_2(\alpha_2)$). This defines a map $\phi_1: \Pi^{-1}(K)\rightarrow U^*_{in}$ (resp. $\phi_2: \Pi^{-1}(R_1(K))\rightarrow U^*_{out}$). It is clear that both $\phi_1$ and $\phi_2$ are symplectic bijections onto their images. 

The restriction of $R_0$ on $\Pi^{-1}(K)$ can be decomposed as
$$R_0|_{\Pi^{-1}(K)}=\phi_2^{-1} \circ \sigma \circ \phi_1.$$

Define a dual lens map $\tilde{\sigma}: U^*_{in}\rightarrow U^*_{out}$ by 
$$\tilde{\sigma}(\alpha):=\left\{
\begin{aligned}
&\phi_2\circ \tilde{R}_0\circ \phi_1^{-1}(\alpha),& &\text{ if } \alpha\in\phi_1(\Pi^{-1}(K)); \\
&\sigma(\alpha), & &\text{ otherwise.}
\end{aligned}
\right.$$
By definition, $\tilde{\sigma}$ coincides with $\sigma$ outside a compact set. Moreover $\tilde{\sigma}\rightarrow \sigma$ in $C^{\infty}$ as $\tilde{R}_0\rightarrow R_0$ in $C^{\infty}$.  It is also clear that $W=K^c$ satisfies $(\ast)$ for some open set $O\subseteq U_{in}^*$. By Corollary \ref{maincor} there exists a Finsler metric $\tilde{\varphi}$ in $D_{1/2}$ agreeing with $\varphi_0$ around the boundary $\partial D_{1/2}$ and the dual lens map for $(D_{1/2},\tilde{\varphi})$ is exactly $\tilde{\sigma}$. We extend $\tilde{\varphi}$ to the whole $\mathbb{T}^n$ by setting it equal to $\varphi_0$ outside $D_{1/2}$. Then it is clear that the first return map is exactly $\tilde{R}_0$. 

If $\varphi_0$ is reversible, we define $\tilde{\sigma}$ by: 
$$\tilde{\sigma}(\alpha)=\left\{
\begin{aligned}
&\phi_2\circ \tilde{R}_0\circ \phi_1^{-1}(\alpha),& &\text{ if } \alpha\in\phi_1(\Pi^{-1}(K)); \\
&-\phi_1 \circ \tilde{R}^{-1} \circ \phi_2^{-1}(-\alpha), & &  \text{ if } \alpha\in -\phi_2(\Pi^{-1}(R_1(K))); \\
&\sigma(\alpha), & &\text{ otherwise.}
\end{aligned}
\right.$$
It is clear that $\tilde{\sigma}$ is symmetric. By Theorem \ref{thm3}, $\tilde{\varphi}$ can be chosen to be reversible. 
\end{proof}
\begin{remark}
The support of the resulting metric perturbation can be made small if the size of $K$ is small.
\end{remark}

\section{Twist maps, minimal configurations and rotation symbols}
\subsection{Twist maps and generating functions}
\begin{definition}[\cite{KH}]
$f: S^1\times (a,b) \rightarrow S^1\times (a,b)$ is \textit{an area-preserving twist map} if:

(i) $f$ is area  and  orientation preserving.

(ii) $f$ preserves boundary components in the sense that there exists an $\epsilon>0$ such that if $(x,y)\in S^1\times (a,a+\epsilon)$ then $f(x,y)\in S^1\times (a,\frac{a+b}{2})$.

(iii) if $F=(F_1,F_2)$ is a lift of $f$ to the universal cover of $\mathbb{R}\times(a,b)$ then $\partial F_1/\partial y>0$.

Here $(a,b)$ can be an open interval or the whole real line.
\end{definition}
If in addition to (i)-(iii) we have

(iv) $f$ twists infinitely at either end. Namely, for all $x\in S^1$ we have
$$\lim_{y\rightarrow a+}F_1(x,y)=-\infty, \lim_{y\rightarrow b-}F_1(x,y)=+\infty,$$
then we say $f$ is \textit{an area-preserving twist map with infinite twist}. The collection of all area-preserving twist maps with infinite twist from $S^1\times (a,b)$ to itself is denoted $IFT(a,b)$. 

Let $F:\mathbb{R}\times (a,b)\rightarrow \mathbb{R}\times (a,b)$ be a lift of $f\in IFT(a,b)$ to the universal cover.  The \textit{generating function} $h(x,x')$ is uniquely characterized by 
$$F(x,y)=(x',y') \Longleftrightarrow y=-\frac{\partial h}{\partial x}(x,x'), y'=\frac{\partial h}{\partial x'}(x,x').$$ 
Notice that if $h$ is $C^2$, the twist condition (iii) is equivalent to $h_{xx'}<0$.

\begin{example}
The map $f_0: S^1\times\mathbb{R}\rightarrow S^1\times\mathbb{R}$ defined by $f(x,y)=(x+y,y)$ is an area-preserving twist map with infinite twist. The generating function is given by
$$h_0(x,x')=\frac{(x'-x)^2}{2}.$$
\end{example}

\begin{example}
Define $f_1: S^1\times(-1,1)\rightarrow S^1\times(-1,1)$ by $f(x,y)=(x+\frac{y}{\sqrt{1-y^2}},y)$. Then $f_1\in IFT(-1,1)$ and the generating function is given by
$$h_1(x,x')=\sqrt{(x'-x)^2+1}.$$
\end{example}

Given a $f\in IFT(a,b)$, if the amount of twisting in (3) has a uniform lower bound $\beta$, then its generating function $h$ will satisfy all the following conditions $(H_1)-(H_{6\theta})$  with $\theta=\cot\beta$ \cite{M87}:
$$h(x,x')=h(x+1, x'+1). \leqno (H_1)$$
$$\lim_{|\xi|\rightarrow\infty}h(x,x+\xi)=+\infty, \text{ uniformly in }x. \leqno (H_2)$$
There exists a positive continuous function $\rho$ on $\mathbb{R}^2$ such that for $x<\xi, x'<\xi'$:
$$h(\xi,x')+h(x,\xi')-h(x,x')-h(\xi,\xi')\geq \int_x^{\xi}\int_{x'}^{\xi'}\rho. \leqno (H_5)$$
$$\left\{
\begin{aligned}
x & \rightarrow & \theta x^2/2-h(x,x') \text{  is convex for any } x'; \\
x' & \rightarrow & \theta x'^2/2-h(x,x') \text{  is convex for any } x.
\end{aligned}
\right. \leqno (H_{6\theta})$$
Here $\theta$ is a positive number. We say $h$ satisfies $(H_6)$ if it satisfies $(H_{6\theta})$ for some $\theta>0$. The conditions $(H_3)$ and $(H_4)$ from \cite{Ban88} can be derived from $(H_5)$ and $(H_6)$. If $h$ is $C^2$, then $(H_5)$ is equivalent to the twist condition $h_{xx'}<0$ and $(H_{6\theta})$ is equivalent to $h_{xx}, h_{x'x'}\leq \theta$. In the twist condition (iii), if $\partial F_1/\partial y$ has a lower bound $\beta$, then the generating function $h$ satisfies $(H_{6\theta})$ with $\theta=\cot\beta$. We use $\mathscr{H}_{\theta}$ to denote the collection of all continuous functions $h:\mathbb{R}^2\rightarrow \mathbb{R}$ satisfying $(H_1)-(H_{6\theta})$.

\subsection{Minimal configuration and rotation symbols}
We refer to \cite{Ban88}\cite{F}\cite{M87}\cite{M88} for the definitions and results we need in the sequel.

A \textit{configuration} is a bi-infinite sequence $\textbf{x}=(...,x_i,...)\in\mathbb{R}^{\mathbb{Z}}$ (with product topology of $\mathbb{R}^{\mathbb{Z}}$). The \textit{Aubry graph} of $\textbf{x}$ is the graph of the piecewise linear function $\Phi: \mathbb{R}\rightarrow\mathbb{R}$ determined by $\Phi(i)=x_i$ at every $i\in\mathbb{Z}$.

Suppose $h$ is a function on $\mathbb{R}^2$ satisfying $(H_1)-(H_6)$.  Define
$$h(x_j,...,x_k):=\sum_{i=j}^k h(x_i, x_{i+1}).$$

A segment $(x_j,...,x_k)$ is said to be \textit{minimal} (for $h$) if it is a minimizer for $h(x_j^*,...,x_k^*)$ with $x_j^*=x_j$ and $x_k^*=x_k$, A configuration is minimal if all its segments are minimal. We use $\mathscr{M}=\mathscr{M}_h$ to denote the set of all minimal configurations. The Aubry graphs of minimal configurations cross at most once. In the survey \cite{Ban88} Bangert shows how minimal geodesics on torus are related to minimal configurations.

A configuration $\textbf{x}'$ is a \textit{translate} of $\textbf{x}$ if there exist integers $j,k$ such that $x'_i=x_{i+j}+k$ for all $i$. We use the notation $T_{(a,b)}$ to denote the translation $T_{(a,b)}\textbf{x}=\textbf{x}'$ where $x'_i=x_{i-a}+b$. 

A translate of a minimal configuration is always minimal. A basic result of Aubry says that the set of translates of a minimal configuration is totally ordered with $\textbf{x}<\textbf{y}$ being defined to be $x_i<y_i$ for all integers $i$. Aubry's result implies that for any minimal configuration $\textbf{x}$, there is a number $\omega=\rho(\textbf{x})$, called the \textit{rotation number} of $\textbf{x}$, such that if $x'_i=x_{i+j}+k$ with $j>0$, then $\textbf{x}'>\textbf{x}$ (resp. $\textbf{x}'<\textbf{x}$) if $j\omega+k>0$ (resp. $j\omega+k<0$).

When $\rho(\textbf{x})$ is irrational, it is also called \textit{rotation symbols} $\tilde{\rho}(\textbf{x})$ of $\textbf{x}$. When $\rho(\textbf{x})=p/q\in\mathbb{Q}, q>0$, we investigate $x'_i=x_{i+q}-p$ i.e. $\textbf{x}'=T_{(-q,-p)}\textbf{x}$.  Notice that \textbf{x} may not be periodic even if $\rho(\textbf{x})$ is rational. We define 
$$\tilde{\rho}(\textbf{x})=\left\{
\begin{aligned}
(p/q) &+& \text{ if } \textbf{x}'>\textbf{x}; \\
p/q & &  \text{ if } \textbf{x}'=\textbf{x};\\
(p/q) &-& \text{ if } \textbf{x}'<\textbf{x}.
\end{aligned}
\right.$$
Namely, $\tilde{\rho}(\textbf{x})=(p/q)+$ (resp. $(p/q)-$) if the Aubry graph of $\textbf{x}'$ is strictly above (resp. below) that of $\textbf{x}$ (see also \cite{M88}). Since minimal configurations cross at most once, the Aubry graphs of $\textbf{x}$ and $\textbf{x}'$ do not cross if $\tilde{\rho}(\textbf{x})=\tilde{\rho}(\textbf{x}')=p/q$. 

\section{An extension of Mather's Destruction of invariant circle}

 Mather \cite{M88}  proved that for any Liouville number $\omega$  and a twist map in $IFT(-\infty,$ $+\infty)$ there exists a $C^{\infty}$-small perturbation with no invariant circle admitting rotation number $\omega$. But the perturbation of Mather is not compactly supported. Nevertheless, we get the following proposition by imitating Mather's construction.  It remains valid in higher dimensional cases under minor modification.

\begin{proposition}\label{mainprop2}
For any $f\in IFT(a,b)$ and any Liouville number $\omega$, we can find a $C^{\infty}$-small perturbation $\tilde{f}\in IFT(a,b)$ and a compact $K\subseteq S^1\times (a,b)$ such that $\tilde{f}-f$ has support $K$ and there is no $\tilde{f}$-invariant circle with rotation number $\omega$.
\end{proposition}

\begin{proof}
We will mainly manifest what modification we make on Mather's construction. Our aim is to prove that for any $r\geq 1$,  one can make a $C^{r+1}$-small perturbation $\tilde{h}$ on the generating function $h$ so that the Peierls' barrier $P_{\omega,\tilde{h}}$ (cf. \cite{M88}) is not vanishing everywhere, which is sufficient to show the absence of $\tilde{f}$-invariant circle with rotation number $\omega$. Notice that $h$ satisfies $(H_1)-(H_5)$. Since we only make perturbation near the invariant circle of $h$ with rotation number $\omega$, we may assume $h\in \mathscr{H}_{\theta}$ for some $\theta$.

The general idea is to firstly choose a rational number $p/q$ close to $\omega$ (we may assume $p/q<\omega$),  and make a $C^{r+1}$-small perturbation $h'$ on the generating function $h$ so that the Peierls' barrier $P_{p/q,h'}$ is positive in some interval $J$ to eliminate the minimal configurations through this interval with rotational symbol $p/q$. Secondly we make an additional $C^{r+1}$-small perturbation $\tilde{h}$ on $h'$ so that $P_{p/q+,h'}$ (or $P_{p/q-,h'}$ if $p/q>\omega$) is positive in an interval $I\subseteq J$. By the modulus of continuity formula  \cite[Theorem 2.2]{M88},  as $p/q$ is sufficiently close to $\omega$ (we can choose such $p/q$ since $\omega$ is Liouville), $P_{\omega,\tilde{h}}$ is positive at some point in $I$ and it finishes the proof.

We now explain how to construct the perturbation of $h$ when $\omega>p/q$.  Suppose $\textbf{x}$ is a minimal configuration in $\mathscr{M}_{p/q}$. Choose an interval $J$ with length$\geq q^{-1}$ in the complement to the set $\{x_i+j\}_{i,j\in\mathbb{Z}}$. Without loss of generality we may assume $J=(x_j, x_k+m)$ for some $j,k,m\in\mathbb{Z}$. For any $\epsilon>0$ and any integer $r\geq 1$, we choose a $C^{\infty}$ nonnegative function $u$ on $\mathbb{R}$ with the following properties:

(a) $u$ has support $\bar{J}$;

(b) $||u||_{C^{r+1}}\leq \epsilon/2$;

(c) $u(\xi)\geq C_1(r)\epsilon/q^{r+1}$, for $\xi\in J'$, here $J'$ is the middle third of $J$ and $C_1(r)$ is a constant depending only on $r$.

Here is how to construct such a function: Define a function $\Psi: \mathbb{R}\rightarrow\mathbb{R}$ by

$$\Psi(t)=\left\{
\begin{aligned}
& \exp\left(\frac{1}{t^2-1}\right),& &\text{ for } |t|<1 \\
& 0,& &\text{ otherwise.}
\end{aligned}
\right.$$

Denote $C_0(r):=||\Psi||_{C^{r+1}}$. Define a function $u_0$ by
$$u_0(t)=\frac{\epsilon}{2^{r+2}q^{r+1}C_0(r)}\Psi(2qt).$$
and let
$$C_1(r)=\Psi(1/3)2^{-r-2}C_0(r)^{-1}.$$
It is not hard to check that $u_0$ satisfies (a)-(c) for $J=(-1/2q,1/2q)$. For a general $J$, we have only to move and rescale $u_0$.

 Define a function $v$ on $\mathbb{R}$ by
$$v(t)=\left\{
\begin{aligned}
& C_2(r)q^{-r-1}\Psi(2q(t-x_{j+1})),& &\text{ for } t\in[x_{j+1}-1/2q, x_{j+1}); \\
& C_2(r)q^{-r-1}\Psi(0),& &\text{ for } t\in[x_{j+1}, x_{k+1}+m); \\
& C_2(r)q^{-r-1}\Psi(2q(t-x_{k+1}-m)),& &\text{ for } t\in [x_{k+1}+m, x_{k+1}+m+1/2q); \\
& 0,& &\text{ otherwise, }
\end{aligned}
\right.$$
where $C_2(r)=2^{-r-1}C_0(r)^{-1}$. Note that $v$ is nonnegative, $C^{\infty}$, supported by an interval with length $\leq 3/q$ and $||v||_{C^{r+1}}=1$.

Now we make a first perturbation on $h$:
$$h'(x,x')=h(x,x')+\sum_{i\in\mathbb{Z}}u(x+i)v(x'+i).$$

We construct $w$ as in \cite{M88}, and set
$$\tilde{h}(x,x')=h'(x,x')+\sum_{i\in\mathbb{Z}}w(x+i)v(x'+i).$$
By a minor modification of the rest of the proof in \cite{M88}, we get a proof of Proposition \ref{mainprop2}.

\end{proof}

\section{Proof of Theorem \ref{thm1}}

We use the setting in Section \ref{sec: perturb_torus} and assume that $\varphi_0(\partial/\partial q_1)=1$,. Namely, the line $t\mapsto (t, q_2)$ is a geodesic for any $q_2$. Denote by $\Lambda:=\varphi_0(-\partial/\partial q_1)$. $\Lambda$ may not be equal to 1 if $\varphi$ is  not reversible. We have $U_p=(-\Lambda, 1)$ and $R_1: S^1\times (-\Lambda,1) \rightarrow S^1\times (-\Lambda,1)$ is given by
$$R_1(q_1, p_1)=(q_1+f'_0(p_1), p_1).$$

\begin{lemma}\label{lem1}
$R_1\in IFT(-\Lambda,1)$. Moreover, $f''_0$ has a positive lower bound.
\end{lemma}

\begin{proof}
We verify the (i)-(iv) in Definition 3.1. (i) and (ii) are clear. Since $\Psi(p_1, -f_0(p_1))$ $\equiv 1,$ by taking derivative with respect to $p_1$, we have
$$\Psi_1-f'_0\Psi_2=0.$$

On the unit circle of $\varphi^*_0$, $(\Psi_1, \Psi_2)(p_1, p_2)=\mathscr{L}^{-1}(p_1, p_2)\in S$, where $S$ is the unit circle of $\varphi_0$, thus $\Psi_1, \Psi_2$ are both bounded. Notice that $p_1\in(-\Lambda, 1)$. As $p_1$ approaches to either ends, $\Psi_2$ goes to $0$ from above. When $p_1\to-\Lambda$, $\Psi_1\to -1/\Lambda$ thus $f'_0\to -\infty$. On the other side, as $p_1\to 1$, we have $\Psi_1\to 1$ and $f'_0\to +\infty$. Hence $R_1$ satisfies (iv).

By taking the second derivative with respect to $p_1$, we have
$$\Psi_{11}-f''_0\Psi_2+(f'_0)^2\Psi_{22}=0.$$
Since $\varphi_0$ is Finsler metric, $\Psi_{11}, \Psi_{22}\geq \delta_1>0$. On $\Gamma_0$ we have $0<\Psi_2\leq \delta_2$, hence $f''_0\geq \delta_1/\delta_2>0$, which implies (iii). 
\end{proof}

Define a function $h:\mathbb{R}^2\to\mathbb{R}$ by
$$h(q, q')=d_{\varphi_0}((q, 0), (q',1)).$$
\begin{lemma}\label{lem2}
$h$ is a generating function of $R_1$, therefore $h\in\mathscr{H}_{\theta}$ for some $\theta$.
\end{lemma}
\begin{proof}
Notice that $h$ only depends on $q'-q$, hence $-h_q=h_{q'}$. Let $\gamma: [a, b]\to \mathbb{R}^2$ be the geodesic from $(q, 0)$ to $(q',1)$. By the first variation formula, $h_{q'}=\mathscr{L}(\gamma'(b))(\partial/\partial q_1)$, which is the $p_1$ coordinate of $\mathscr{L}(\gamma'(b))\in UT^*_{(q',1)}\mathbb{R}^2$. 
\end{proof}

\begin{proof}[Proof of Theorem \ref{thm1}]
For any Liouville number $\omega$ and any $\epsilon>0, r\geq 1$, choose $p/q$ sufficiently close to $\omega$ as in the proof of Proposition \ref{mainprop2}. We take $J=(-1/2q, 1/2q)$ and construct $\tilde{h}:\mathbb{R}^2\rightarrow\mathbb{R}$ with $||\tilde{h}-h||_{C^{r+1}}\leq \epsilon$ and the twist map $\tilde{R}_1\in IFT(-\Lambda, 1)$ associated to $\tilde{h}$ has no invariant circle with rotational number $\omega$. From Aubry-Mather theory, the absence of $\tilde{R}_1$-invariant circle implies the existence of a minimal $\tilde{R}_1$-invariant Cantor set whose projection to $S$ is also Cantor.

Let $K$ be the support of $\tilde{R}_1-R_1$. From the construction in the proof of \ref{mainprop2} we know that $\Pi_q(K)=B(1/2q)$ and $\Pi_q(R_1(K))\subseteq B(\omega+\frac{3}{q})\subseteq B(2q)$ for large $q$.  By Proposition \ref{mainprop1} there exists a  Finsler metric $\tilde{\varphi}$ on $\mathbb{T}^2$ such that the Poincar\'{e} map of the geodesic flow is $\Pi^{-1}\circ \tilde{R}_1\circ \Pi$. The $\tilde{R}_1$-invariant Cantor set with rotation number $\omega$ implies the existence of a nondense geodesic with rotation vector $(\omega, 1)$. 

When $\varphi_0$ is reversible, the perturbation of the dual lens map can be made reversible, thus the reversible version of Proposition \ref{mainprop1} can be applied to get reversible perturbation of $\varphi_0$ with the desired nondense geodesic. This finishes the proof of Theorem \ref{thm1}.
\end{proof}

\begin{remark}\label{rem: conj_pts}
If a Finsler 2-torus has no conjugate points, any Liouville number $\omega$ gives us a foliation of the torus whose leaves are geodesics with the same rotation vector colinear to $(\omega,1)$, hence the Peierls' barrier $P_{\omega}\equiv 0$.  Since in our example,  $P_{\omega}$ is not vanishing everywhere,  the Finsler torus we get in Theorem \ref{thm1} do have conjugate points even though it is almost flat.
\end{remark}

\end{document}